\documentclass[letterpaper, 10 pt, conference]{ieeeconf}  

\IEEEoverridecommandlockouts                              
\usepackage{amsmath,amssymb,amsfonts,mathtools,dsfont}
\usepackage[margin = 0.8in]{geometry}
\usepackage{algorithm}
\usepackage[noEnd=false]{algpseudocodex}
\usepackage{booktabs}
\usepackage{subcaption}
\usepackage{comment}
\usepackage{enumerate}

\algnewcommand\algorithmicinput{\textbf{Input:}}
\algnewcommand\Input{\item[\algorithmicinput]}
\algrenewcommand\algorithmicoutput{\textbf{Output:}}
\algrenewcommand\Output{\item[\algorithmicoutput]}

\newtheorem{assumption}{Assumption}
\newtheorem{proposition}{Proposition}
\newtheorem{remark}{Remark}

\newcommand{\tp}{^\top}
\newcommand{\inv}{^{-1}}
\newcommand{\R}{\mathbb{R}}
\renewcommand{\S}{\mathcal{S}}
\newcommand{\simulator}{\S}
\newcommand{\ham}{\mathcal{H}}
\newcommand{\half}{\frac{1}{2}}
\DeclareMathOperator*{\argmin}{arg\,min}

\newcommand{\normal}{\mathcal{N}}

\newcommand{\ud}{\mathrm{d}}
\newcommand{\G}{\mathcal{G}}
\newcommand{\grad}{\nabla}

\newcommand{\backward}[1]{\overset{\leftarrow}{#1}}
\newcommand{\one}{\mathds{1}}
\newcommand{\diag}{\mathrm{diag}}
\newcommand{\id}{\mathcal{I}}
\renewcommand{\L}{\mathcal{L}}

\title{\LARGE \bf
A Dual Ensemble Kalman Filter Approach to Robust Control of Nonlinear Systems: An Application to Partial Differential Equations
}

\author{Anant A. Joshi, Saviz Mowlavi, Mouhacine Benosman
\thanks{A. A. Joshi (anantaj2@illinois.edu) is with the Coordinated Science Laboratory and the Department of Mechanical Science and Engineering at the University of Illinois Urbana-Champaign, Urbana, IL 61801, USA. S. Mowlavi (mowlavi@merl.com) is with Mitsubishi Electric Research Laboratories (MERL), Cambridge, MA 02139, USA. Mouhacine Benosman (m\_benosman@ieee.org) is with Amazon Robotics, North Reading, MA 01864, USA. This work was completed during A. A. Joshi’s internship at
MERL and prior to M. Benosman joining Amazon Robotics.}%
}

\begin{document}

\maketitle
\thispagestyle{empty}
\pagestyle{empty}

\begin{abstract}

This paper considers the problem of data-driven robust control design for nonlinear systems, for instance, obtained when discretizing nonlinear partial differential equations (PDEs). A robust learning control approach is developed for nonlinear affine in control systems based on Lyapunov redesign technique. The robust control is developed as a sum of an optimal learning control which stabilizes the system in absence of disturbances, and an additive Lyapunov-based robustification term which handles the effects of disturbances. The dual ensemble Kalman filter (dual EnKF) algorithm is utilized in the optimal control design methodology. A simulation study is done on the heat equation and Burgers partial differential equation.  

\end{abstract}

\section{Introduction}
In this paper, we are primarily interested in the robust control of nonlinear affine in control systems modelled as 
\begin{align}
\dot{x}(t) = a(x(t)) + b(x(t))u(t) + d(t,x), \quad x(0) = x_0
\label{eq:control-system}
\end{align}
where $x(t) \in \R^n$, $u(t) \in \R^m$ and $d(t,x) \in \R^l$ for all $(t,x)$, and $d$ is regarded as an unknown disturbance bounded in norm by a known real valued function $\lambda(t,x)$. The emphasis is on systems where we do not have explicit access to $a$ and $b$, but we can access trajectories of the system, either generated by a simulator or collected from real-life experiments. For simplicity, we focus in the remaining of the paper on the case of simulated trajectories. 
Furthermore, we target the application of our algorithms to the control of PDEs, when their discretized model (\ref{eq:control-system}) is available as a simulator. 

Indeed, PDE control is challenging because a PDE is by nature infinite dimensional and may have strong nonlinearities. 
The reduce-then-design approach is a well researched control methoology, which involves discretization of the PDE followed by dimensionality reduction, which yields a model amenable to the application of standard model based control approaches \cite{ 
leibfritz2007numerical, hovland2008explicit, barbagallo2009closed,sipp2016linear, 
tsolovikos2020estimation}. There is a recent body of work using data driven methods for building more accurate reduced order models of the PDE \cite{
bhattacharya2021model, fresca2021comprehensive, kaiser2021data}. While more accurate reduced order models are beneficial to the performance of the controller, their complexity makes them computationally challenging to implement. 

In this work, we focus more on the design of the controller with emphasis on obtaining the control policy with only simulator for the model, without access to the actual model parameters. See for example \cite{duriez2017machine,blanchard2021bayesian,fan2020reinforcement, Garnieretal21} for some other recent efforts in obtaining data-driven controllers using a simulator of the PDE.   
In particular, we work on the stabilization problem, which aims to drive the PDE state to zero, by posing it as an optimal control problem. The first step is to discretize the PDE in space, to yield a high-dimensional nonlinear system of ordinary differential equations (ODEs). We consider two different cases for the simulator availability: 
\begin{enumerate}
    \item a linear reduced order simulator, obtained using Dynamic Mode Decomposition with control (DMDc) \cite{proctor2016dynamic,zhang-2024} is available.
    \item an extension to the case where a simulator for the full high-dimensional nonlinear discretized model of the PDE is available.
\end{enumerate}

We build on previous work in this area \cite{zhang-2024}  by taking a robust learning control approach. We construct the robust control as a sum of two parts: an optimal learning controller, which we expect to produce stabilization in the absence of disturbances, and an additional term that builds on the optimal control term and is based on the Lyapunov redesign approach \cite{khalil} to suppress the effect caused by disturbances. 
The optimal control is approximated using the dual ensemble Kalman filter (dual EnKF) algorithm \cite{joshi-2022}, which simulates multiple interacting copies of the system.

The ensemble Kalman filter (EnKF) \cite{yang-2016,amir-2019} has historically been an algorithm used for filtering, and is a key numerical method especially for high-dimensional systems, for example in weather prediction \cite{evensen2006} (see \cite{bishop-2023} for more references). It features the design of an interacting particle system to sample from the posterior density of a filtering problem to provide a state estimate. 
The dual EnKF algorithm, inspired from the EnKF, computes the optimal control by converting the control problem into the problem of sampling from an appropriate probability density, through the log-transform duality between optimal control and filtering \cite{fleming-1982}. Such an approach for solving optimal control by posing it as a sampling problem is well studied
(see \cite{todorov-2007,kappen-2005,rawlik2013stochastic,hoffman-2017,levine-2018}). 
The novelty of the dual EnKF lies in the design of an interacting particle system, inspired from the EnKF, to solve the sampling problem. 
The dual EnKF-based control exhibits a distinct advantage in its performance on systems of high dimension -- a trait inherited from the EnKF. It performs almost two orders of magnitude faster when compared with policy gradient type approaches \cite{mihailo-2021-tac,fazel-2018-mlr} as was computationally demonstrated in \cite{joshi-2022}.

The paper is organized as follows. In Section \ref{sec:linear} we introduce and solve the problem for the case when the system is linear time invariant (LTI) to clarify the main ideas. Then we present the extension to the nonlinear affine in control case in Section \ref{sec:nonlinear}. 
Finally, in Section \ref{sec:pde} we present an application to the stabilization of PDEs, with the heat equation and Burgers' equation as examples. 





\textbf{Notation:} 
$|\cdot|$ denotes the Euclidean norm, $\tp$ denotes the transpose of a matrix, and $\id$ refers to the identity matrix.

\section{Solution for linear reduced order simulator}
\label{sec:linear}

Consider the simplified case when \eqref{eq:control-system} is a linear time invariant (LTI) system:
\begin{align}
\dot{x}(t) = Ax(t) + Bu(t) + d(t,x)
\label{eq:lin-control-system}
\end{align}
where as earlier, $x(t) \in \R^n$, $u(t) \in \R^m$ and 
$d(t,x) \in \R^l$ for all $(t,x)$, and $d$ is regarded as a disturbance. From here on we may suppress the $t$ argument to improve readability. We construct the robust control as a sum of a stabilizing control and an additional robustification term. We first present a recipe to obtain the stailizing control, and then to obtain the robust control, and lastly a simulator based method to obtain the two. 

\subsection{Stabilizing control}

Consider, for any arbitrary but fixed $T > 0$, the optimal control problem
\begin{subequations}
\begin{align}
\min_{u(\cdot)} & \bigg( x(T) \tp G x(T) + \half\int_{0}^{T} \underbrace{|Cx(t)|^2 + u(t)\tp R u(t)}_{=: \L(x(t),u(t))} \ud t \bigg)  \\
\text{ s.t. } & \text{system }\eqref{eq:lin-control-system} \text{ with zero disturbance, that is, } d \equiv 0
\end{align}
\label{eq:LQR}
\end{subequations}
where $C \in \R^{n_1 \times n}$ (for arbitrary $n_1 > 0$), $G \in \R^{n \times n}$ and $R \in \R^{m \times m}$, and  define $Q := C\tp C$. 
\begin{assumption}
We make the following assumptions about the structure of the optimal control problem \eqref{eq:LQR}
\begin{enumerate}[(i)]
\item $(A,B)$ is controllable and $(A,C)$ is observable
\item $R,G \succ 0$
\end{enumerate}

\end{assumption}
Then there exists a positive definite solution $\{P(t) : t \in [0,T]\}$ to the differential Riccati equation (DRE) \cite[Chapter 3]{k-sivan}:
\begin{align*}
-\dot{P} = A\tp P + PA - PBR\inv B\tp P + Q, \quad P_T = G
\end{align*}
which converges to $\bar{P} \succ 0$ as $T \to \infty$ and $\bar{P}$ solves the algebraic Riccati equation (ARE)
\begin{align*}
0 = A\tp \bar{P} + \bar{P}A - \bar{P}BR\inv B\tp \bar{P} + Q.
\end{align*}
Moreover, the control $u = -\bar{K}x$ with $\bar{K} := R\inv B\tp \bar{P}$ makes the system asymptotically stable \cite[Theorem 3.7]{k-sivan}. 

\subsection{Robust control}

Suppose that for the system \eqref{eq:lin-control-system} in the case of no disturbance, there exists a stabilizing control $u = -\bar{K}x$ and a strictly positive definite Lyapunov function $V(x) = \half x\tp \bar{P} x$ with $\dot{V} = -x\tp \bar{P}(A - B\bar{K})x \le 0$ along the controlled trajectories with zero disturbance. We let $\bar{K}$ be the gain obtained from the optimal control demonstrated previously and $\bar{P}$ be the solution of the ARE.  

We are interested in the idea of disturbance rejection using Lyapunov redesign \cite[Chapter 14]{khalil}. The idea is to design a robust control $u_d$ such that the controller $u = -\bar{K}x + u_d$ makes the system \eqref{eq:control-system} asymptotically stable in the presence of disturbance. To that end, we make the following assumption, and then present a design methodology for $u_d$. 

\begin{assumption}
\begin{enumerate}[(i)]
    \item The rank of $B$ is $n$. 
    \item There exists a known $\lambda$ such that $0 \le |d(t,x)| < \lambda(t,x) < \infty$ for each $(t,x) \in [0,\infty) \times \R^n$.
\end{enumerate}
\label{assn:B}
\end{assumption}
\begin{remark} 
    Assumption \ref{assn:B}-(i) is required to motivate the theoretical derivation for the linear system, but we relax it in our implementations for PDE control. 
\end{remark}

Consider the controller $u = -\bar{K}x + u_d$ with 
\begin{align*}
    u_d := - \frac{\lambda(t,x)}{| \bar{P}x|} B^{\dagger} \bar{P}x, \quad B^{\dagger} := (B\tp B)\inv B\tp,
\end{align*}
where $B^{\dagger}$ denotes the Moore-Penrose pseudoinverse. 
The following result demonstrates the effectiveness of the robust control.
%
%

\begin{proposition}
Using the control $u = -\bar{K}x + u_d$ renders the system \eqref{eq:control-system} asymptotically stable. 
\end{proposition}
\begin{proof}
We follow the method in \cite[Chapter 14.2]{khalil}, using the same Lyapunov function $V(x) = \frac{1}{2} x\tp \bar{P}x$ as before. 
The quadratic form $V$ is strictly positive definite (by assumption) hence is a valid Lyapunov function. Taking derivative along system trajectories, 
\begin{align*}
\dot{V}(x) &= -x\tp \bar{P}(A - B\bar{K})x  +  (Bu_d + d)\tp \bar{P} x \\
& \le  | d | \cdot | \bar{P}x| -\lambda |\bar{P} x| \\
& \le  (|d| - \lambda ) \, | \bar{P}x| < 0.
\end{align*}
For the first inequality we recall that $-x\tp \bar{P}(A - B\bar{K})x \le 0$ and use Cauchy-Schwarz inequality. 
Next by properties of the Moore-Penrose pseudoinverse, $BB^{\dagger}\bar{P}x$ is the orthogonal projection of $\bar{P}x$ onto the column span of $B$. 
Hence, under Assumption \ref{assn:B}-(i), we have $BB^{\dagger}\bar{P}x = \bar{P}x$ therefore $Bu_d = - \frac{\lambda}{| \bar{P}x|} \bar{P}x$.
\end{proof}
 \begin{remark}
     When implementing $u_d$ we add a regularizing parameter to avoid division by zero, which makes the system asymptotically stable till it enters a ball around the origin, i.e., practical stability. 
 \end{remark}

\subsection{Data-driven (Simulator-based) implementation}
In this section, we present a method to implement the control $u = -\bar{K}x + u_d$ with only access to a disturbance-free system simulator of \eqref{eq:control-system}. We use the dual ensemble Kalman filter (dual EnKF) algorithm \cite{joshi-2022} for the same. To use the algorithm, we need the following assumptions:
\begin{assumption}
We have access to the following:
\begin{enumerate}[(i)]
\item Knowledge of optimal control matrices $Q,R,G$. 
\item Simulator of the dynamical system with no disturbance, that is, we have access to function evaluations of $\S(x,u) := Ax + Bu$. Moreover, we assume we can run the simulator backward in time, that is, to find a trajectory of the system by specifying the terminal condition. 
\end{enumerate}
\label{assn:access}
\end{assumption}
\begin{remark}
\label{rmk:evaluate}
With access to a perfect simulator, one may exactly find the model matrices $A$ and $B$ in $n+m$ evaluations of the simulator (set $u=0$ and evaluate the simulator at the basis vectors of $\R^n$ to find $A$ and equivalently for $B$). However, the utility of the simulator is revealed when we consider the nonlinear case, especially for high-dimensional systems like partial differential equations, where estimating the state dynamics in this manner is not possible. 
\end{remark}

The emphasis is on obtaining the controller in a model-free way. The following three steps are done (which are elaborated upon after listing them):
\begin{enumerate}
\item Find an approximation to $\bar{P}$, the solution of the ARE, using the dual EnKF algorithm \cite{joshi-2022}. See Appendix \ref{app:linenkf} for details
\item Find an approximation $\bar{u}^{(N)}$ for $\bar{u} := -\bar{K}x$ using \cite[Algorithm 2]{joshi-2022} (recalled in Appendix \ref{app:ubar})
\item Find an approximation $u_d^{(N)}$ for $u_d$ using Algorithm \ref{alg:ud}
\end{enumerate}

\textbf{Step 1:}
Using the simulator, we find $\bar{P}^{(N)}$, an approximation to the solution of the ARE $\bar{P}$ by running the dual EnKF algorithm of \cite{joshi-2022}, which simulates $N$ copies of the dynamical system  along with a mean-field coupling term to approximate the solution of the ARE. The algorithm is provided in Appendix \ref{app:linenkf}. 

\textbf{Step 2:} 
To help evaluate the optimal control in a model-free way, define the Hamiltonian,
\begin{align}
\ham(x,u) &:= (\bar{P}^{(N)}x)\tp(Ax + Bu) + \half(x\tp Q x + u\tp R u) \nonumber
\\ &= (\bar{P}^{(N)}x)\tp \S(x,u) + \half\L(x,u) \label{eq:ham-lin}
\end{align}
The Hamiltonian is constructed so that it can be evaluated using function calls of $\S$. For a fixed $x$, it is a quadratic function of $u$ with the unique minima at the optimal control. 
Therefore, $\bar{u} = \argmin_{u} \ham(x,u)$ and the minimization can be carried out using gradient estimation as shown in \cite[Algorithm 2]{joshi-2022} (recalled in Appendix \ref{app:ubar}) or using zero-order methods, such as \cite{bach-2016}. 

\textbf{Step 3:}
Similarly, to find $u_d$ we solve $u_d = -\lambda \argmin_{u} | Bu - \frac{\bar{P}x}{| \bar{P}x|} |$ in Algorithm \ref{alg:ud}.  If $B$ is known, one may directly use the Moore-Penrose pseudo inverse. If $B$ is unknown, one may use zero order optimization methods \cite{bach-2016} where $Bu = \S(0,u)$ can be obtained using only access to simulator or one may estimate $B$ as specified in Remark \ref{rmk:evaluate} and use the pseudo inverse. 
{Using zero order optimization is  preferred over estimating $B$ in cases when the simulator is very high dimensional. }

\begin{algorithm}
\begin{algorithmic}[1]
\Input System state $x$, regularizing parameter $r$, robust gain $\lambda$, $\bar{P}^{(N)}$
\State $r_1 := \max( \, |\bar{P}^{(N)}x| \, , \, r \, ) $
\If{$B$ is known}
\State $B^{\dagger} := (B\tp B)\inv B\tp$
\State $v^{(N)} := r_1\inv B^{\dagger} \bar{P}^{(N)}x$
\ElsIf{$B$ is unknown}
\State $v^{(N)} := \argmin_{v} | \simulator(0,v) - \frac{\bar{P}^{(N)}x}{r_1} |$
\EndIf
\State \Return $u^{(N)}_d := -\lambda v^{(N)}$
\end{algorithmic}
\caption{Algorithm to find $u_d$}
\label{alg:ud}
\end{algorithm}


\section{Extension for nonlinear discretized model of the PDE}
\label{sec:nonlinear}

Let us consider now the main result of this paper, dealing with the case of the nonlinear affine in control system \eqref{eq:control-system}. We will extend the robust control design methodology presented in Section \ref{sec:linear} to \eqref{eq:control-system}. Similar to the previous section, we will obtain the robust control as a sum of a stabilizing control and a robustification term, and then present a simulator based methodology to obtain both terms. 
\subsection{Stabilizing control}
In an effort to find a stabilizing control $\bar{u}$, we consider, for any arbitrary but fixed $T > 0$, the optimal control problem
\begin{subequations}\label{eq:nLQR}
\begin{align}
\min_{u(\cdot)} & \bigg( \mathcal{G}(x(T)) + \half \int_{0}^{T} \underbrace{c(x(t)) + u(t)\tp R u(t)}_{=: \L(x(t),u(t))} \ud t \bigg) \\
\text{ s.t. } & \text{system }\eqref{eq:control-system} \text{ with zero disturbance, that is, } d \equiv 0
\end{align}
\end{subequations}
where $c,\G$ are non-negative real valued function, and $R \in \R^{m \times m}$ is symmetric and strictly positive definite. The value function for the problem is defined as the cost-to-go,
\begin{align*}
    \phi(s,x) &:= \min_{u(\cdot)}  \bigg( \mathcal{G}(x(T)) + \half \int_{s}^{T} \L(x(t),u(t)) \ud t \bigg) \\
\text{ s.t. } & \text{system }\eqref{eq:control-system} \text{ with } X_s = x \text{ and } d \equiv 0
\end{align*}
and it satisfies the Hamilton Jacobi Bellman partial differential equation. The optimal control is computed as $\bar{u}(t,x) = -R\inv b(x)\tp \grad \phi(t,x)$ \cite{liberzon}. 
\subsection{Robust control}
\begin{assumption}
\begin{enumerate}[(i)]
    \item There exists a control law $\bar{u}$ which makes \eqref{eq:control-system} with zero disturbance asymptotically stable. Moreover, there exists a strictly positive definite Lyapunov function $V$ with $\dot{V} < 0$ along the controlled trajectories with zero disturbance. 
    \item The rank of $b(x)$ is $n$ for all $x \in \R^n$. 
    \item There exists a known $\lambda$ such that $0 \le |d(t,x)| < \lambda(t,x) < \infty$ for each $(t,x) \in [0,\infty) \times \R^n$.
\end{enumerate}
\label{assn:b}
\end{assumption}
\begin{remark}
    Assumption \ref{assn:b}-(ii) is required for the theoretical result, but will be relaxed in the PDE control implementation. 
\end{remark}


Similar to the linear case, consider the controller $u = \bar{u} + u_d$ with 
\begin{align*}
    u_d := - \frac{\lambda}{| \grad V |} b^{\dagger} \grad V, \quad b^{\dagger} := (b\tp b)\inv b\tp.
\end{align*}


\begin{proposition}
Using the control $u = \bar{u} + u_d$ renders the system \eqref{eq:control-system} asymptotically stable. 
\end{proposition}
\begin{proof}
We follow the method in \cite[Chapter 14.2]{khalil}. 
Taking the derivative of $V$ along system trajectories, 
\begin{align*}
\dot{V}(x) &= \grad V (x(t)) \tp \bigg( a(x(t)) + b(x(t))\bar{u}(t) \\ & \quad + b(x(t))u_d(t) + d(t,x) \bigg) \\
& \le  (| d | \cdot | \grad V(x)| -\lambda |\grad V (x)|) \\
& \le (|d| - \lambda ) |\grad V (x)| < 0.
\end{align*}
For the first equality we recall Assumption \ref{assn:b} and Cauchy-Schwarz inequality. 
Moreover, under Assumption \ref{assn:b}-(ii), $bu_d = - \frac{\lambda}{| \grad V(x)|} \grad V(x)$.
\end{proof}

\subsection{Data-driven (Simulator-based) implementation}

To approximate $\bar{u}$ and $u_d$ we use an approach similar in spirit to Section \ref{sec:linear}. First, we use the nonlinear dual EnKF algorithm \cite{joshi-2022} to approximate the gradient of value function $\grad \phi^{(N)}(x)$. Implementation details can be found in Appendix \ref{app:nlenkf}.
Then we define the nonlinear counterpart of the Hamiltonian  
\begin{align} \label{eq:ham-nl}
\ham(x,u) &:= (\grad \phi(x)^{(N)})\tp \S(x,u) + \half\L(x,u)
\end{align}
where the simulator is now nonlinear, that is, $\S(x,u) = a(x) + b(x)u$. The Hamiltonian can again be evaluated using function calls of the simulator. Moreover, it is quadratic in the control, and can be minimized using \cite[Algorithm 2]{joshi-2022} (recalled in Appendix \ref{app:ubar}), or zero order optimization approaches \cite{bach-2016}. Similar to the linear control case, to find $u_d$ we use $u_d = -\lambda \argmin_{u} | b(x)u - \frac{\grad V(x)}{| \grad V(x)|} |$ as given in Algorithm \ref{alg:ud} (replacing $\bar{P}^{(N)}x$ by $\grad V^{(N)}(x)$) where if $b$ is not known, the optimization can be done by zero-order approaches \cite{bach-2016} or by estimating $b$ similar to Remark \ref{rmk:evaluate}.

\section{Application to forced nonlinear PDEs}
\label{sec:pde}
We consider PDEs of the form
\begin{align*}
    \frac{\partial z}{\partial t}(t,y) + \mathcal{F}(z(t,y)) = \omega(t,y),
\end{align*}
where $\mathcal{F}$ is a differential operator that specifies the structure of the PDE, $z$ is the state of the PDE and $\omega$ is the external input. The functions $z, \omega : \R_+ \times [0,L] \to \R$. Mathematically, the goal of stabilization is to make 
\begin{align}\label{eq:L2objective}
\lim_{t \to \infty}\|z(t,\cdot)\|_{L^2} := \lim_{t \to \infty}\int_{0}^{L} |z(t,y)|^2 \ud y = 0.    
\end{align}
For numerical implementation, we consider a time interval of $[0,T]$ and discretize the PDE in space on a grid of $p$ uniformly-spaced points in $[0,L]$ to obtain $z_p(t) \in \R^p$ for each $t \ge 0$, so that the PDE is reduced to a set of $p$ ODEs. We choose $m$ basis functions $\{\chi_j\}_{j=1}^{m}$ to discretize the control as 
$\omega(x,t) = \sum_{j=1}^m \chi_j(x) U_j(t)$ where $\chi_j : [0,L] \to \R$ is the indicator function of $[\frac{j-1}{m},\frac{j}{m}]$ and $U_j : [0,T] \to \R$. 
Then $U := (U_1,U_2,\ldots,U_m)$ is interpreted as the control. The discretized PDE has nonlinear affine in control form,
\begin{align}
\frac{\ud z_p}{\ud t} + F(z_p) = BU
\label{eq:discrete-pde}
\end{align}
where $F$ approximates the derivatives and $B$ obtained by discretizing $\{\chi_j\}_{j=1}^{m}$. To study the effect of disturbances, we let $U(t) = u(t) + d(t)$ where $u(t)$ is the control action applied to the system and $Bd(t)$ is the effective disturbance acting on the system. We consider here disturbances added directly to the control, which is a special case of the theory presented earlier. In the simulations, we consider only time varying disturbances, so we denote disturbance as $d(t)$. 
We use the Controlgym library in Python \cite{zhang2023controlgym} for numerical implementation. 
Now we first recall the heat equation and Burgers equations, give some information about the implementation of both PDEs, then go on to discuss the simulation results.
\subsection{Heat equation}

The heat equation is given by 
\begin{align*}
    \frac{\partial z}{\partial t}(t,y) - \nu \frac{\partial^2 z}{\partial y^2}(t,y) = \omega(t,y).
\end{align*}
The value of $\nu = 0.002$ is used. Since the PDE is linear, upon discretization, \eqref{eq:discrete-pde} reduces to an LTI system, hence we use the robust control method design discussed in Section \ref{sec:linear} to implement the robust control. The dimension of the control used is $m=10$, while the dimension of the state is $n=100$, thus relaxing Assumption \ref{assn:B}-(i) in the implementation.  Other simulations details can be obtained in Appendix \ref{app:heat}.
\begin{figure}
\centering
    \begin{subfigure}[b]{0.45\textwidth}
        \centering
        \includegraphics[scale = 0.35]{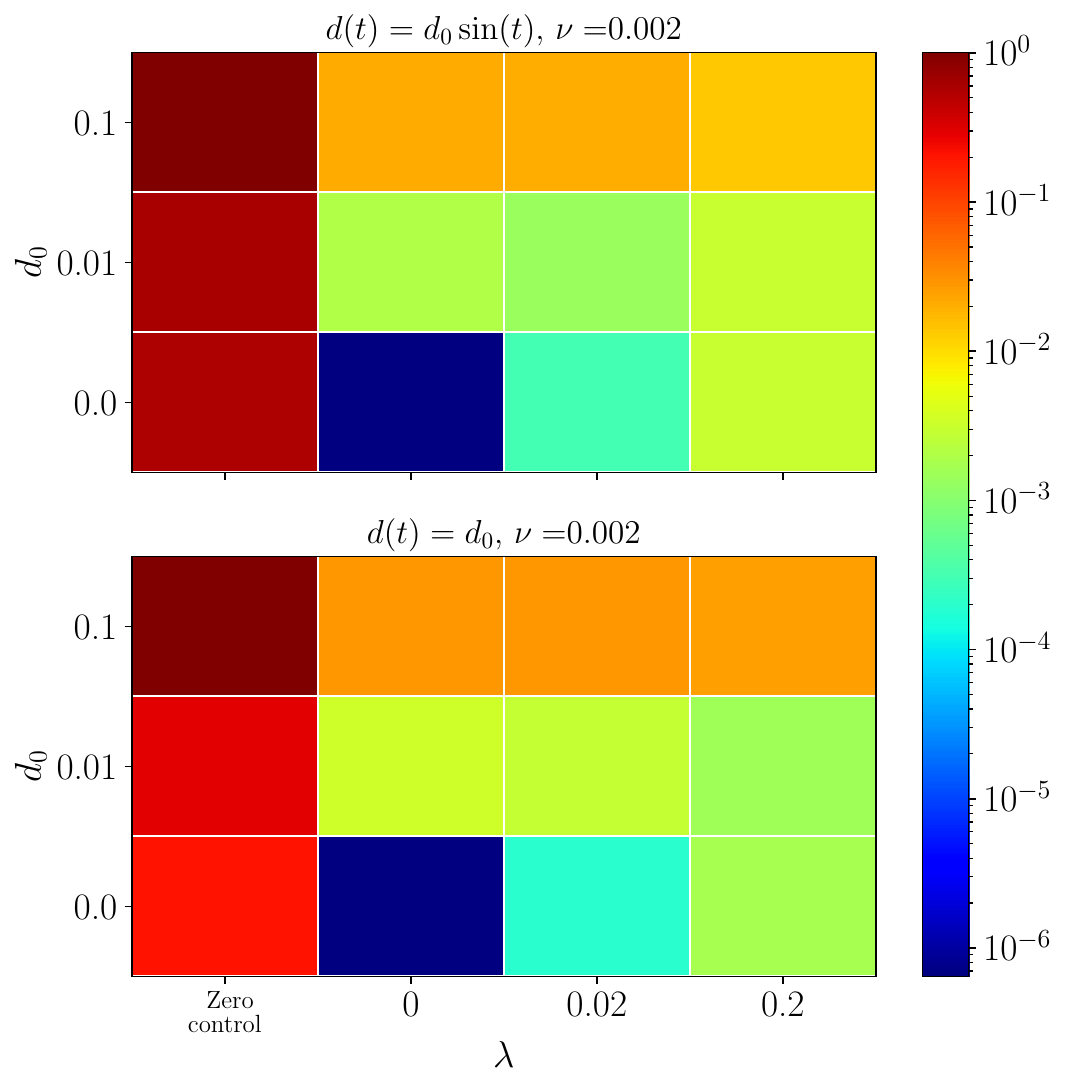}
        \caption{}
        \label{fig:heatmap-heat}
    \end{subfigure}
    \begin{subfigure}[b]{0.45\textwidth}
        \centering
        \includegraphics[scale = 0.35]{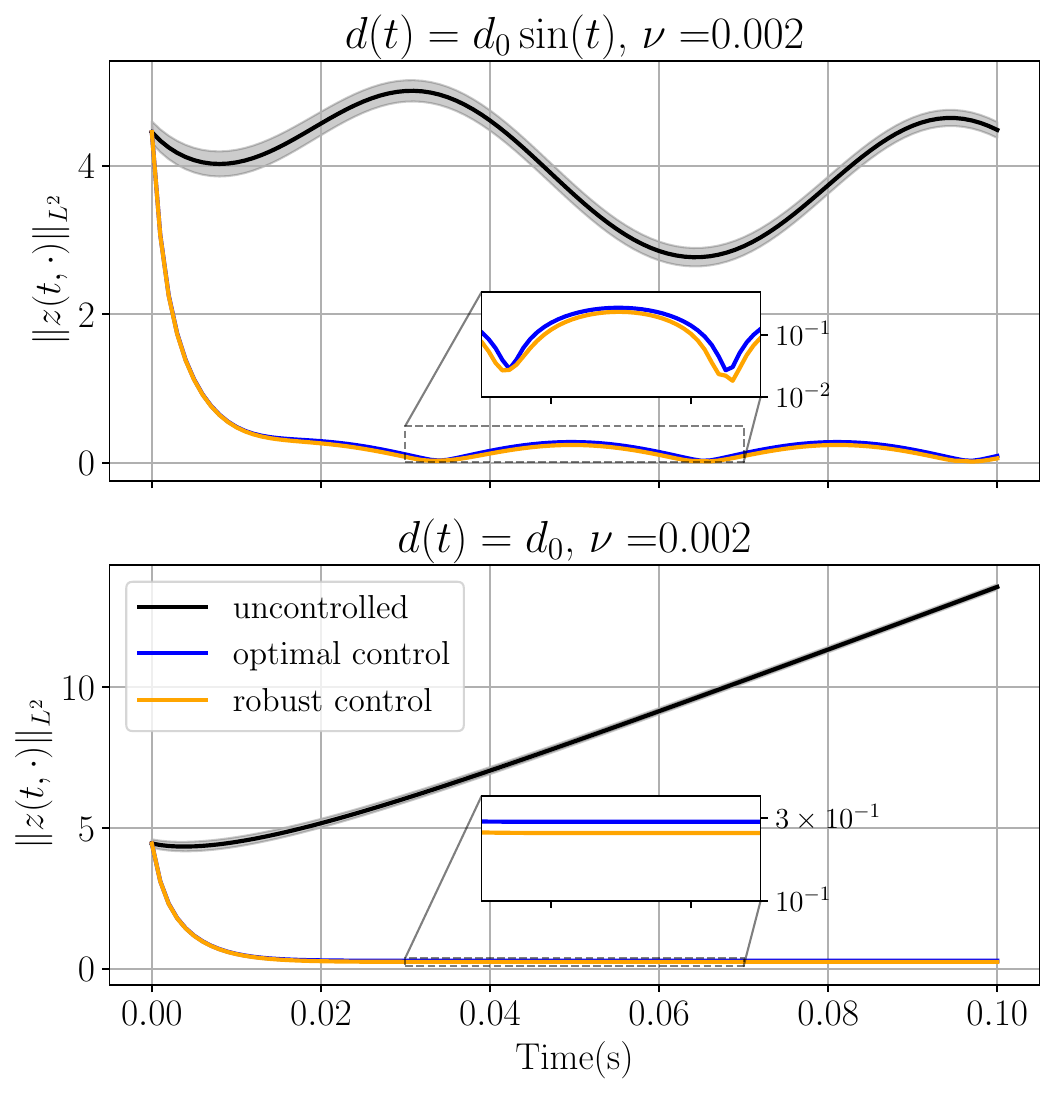}
        \caption{}
        \label{fig:time-heat}
    \end{subfigure}
\caption{Results for control of heat equation. (a) Mean of $\frac{\|z(T,\cdot)\|_{L^2}}{\|z(0,\cdot)\|_{L^2}}$ over 100 different simulations (b) Mean and variance of $\|z(t,\cdot)\|_{L^2}$ for $d_0=0.1, \lambda=0.2$. }
\label{fig:heat}
\end{figure}


\subsection{Burgers' equation}

Recall the Burgers' equation, 
\begin{align*}
\frac{\partial z}{\partial t}(t,y) + z(t,y) \frac{\partial z}{\partial y}(t,y) - \nu \frac{\partial^2 z}{\partial y^2}(t,y) = \omega(t,y).
\end{align*}



We present two sets of results to demonstrate the performance of the robust control algorithm developed in the paper when applied to the Burgers equation. In the first set of results, we calculate the robust control using a simulator of the linear reduced order model of the PDE, and in the second set of results, the robust control is calculated using a simulator of the full nonlinear discretized PDE. In both results, the obtained control is tested by applying it on the full nonlinear discretized PDE. We study both sets of results for two values of viscosity, $\nu = 0.02,0.002$.

\noindent \textbf{Control design using linear reduced model:} 
We perform a model reduction on the Burgers' equation to obtain a linear system using the Dynamic Mode Decomposition for control (DMDc) algorithm \cite{zhang-2024}. The DMDc yields a projection matrix $\Phi \in \R^{n \times p}$ with orthogonal rows, where $n$ is the dimension of the reduced state (usually $n \ll p$), and an LTI system  
\begin{align}
\dot{x} = Ax + Bu
\label{eq:DMDc}
\end{align}
where $x \in \R^n$ is the reduced state and satisfies the relation $x = \Phi z_p$, and $u \in \R^m$ where as described, $n$ and $m$ are chosen by the user. The discretized PDE state can be reconstructed as $z_p = \Phi\tp x$. DMDc yields a discrete-time system, but we transform it into its continuous-time equivalent, since we design control in the continuous time domain. 

Using a simulator for the obtained linear system \eqref{eq:DMDc}, we get an approximation to the optimal stabilizing LQ controller, with the linear EnKF methodology described in Section \ref{sec:linear}. 
That is, we choose optimal control weights $Q,R,G$ (as described in \eqref{eq:LQR}) and find $\bar{P}^{(N)}$ for the LTI system \eqref{eq:DMDc}. The stabilizing control is calculated as $u = -R\inv B\tp \bar{P} x$ where $x = \Phi z_p$. We use the robust control methodology described in Section \ref{sec:linear} to design $u_d$.

\noindent \textbf{Control design using full nonlinear model:}
Using a simulator for the full nonlinear PDE \eqref{eq:discrete-pde}, we obtain a stabilizing control $\bar{u}$ and robust controller $u_d$ with the nonlinear dual EnKF methodology discussed in Section \ref{sec:nonlinear}. We choose $c(x) := |x|^2$ and $\mathcal{G}(x) = |x|^2$ in \eqref{eq:nLQR}. 
The dimension of the state used is $p=128$ and the dimension of the control is $m=10$, thus relaxing Assumption \ref{assn:b}-(ii) in the implementation. 


\subsection{Discussion of results}
We study the effect of two types of disturbances, sinusoidal with $d(t) := d_0\sin(t)$ and constant $d(t) := d_0$. The $\lambda$ function is chosen as a constant function. The obtained control is applied to the full discretized nonlinear PDE \eqref{eq:discrete-pde}. 
Both PDEs are initialized for 100 randomly sampled iid initial conditions (inspired from the previous work \cite{zhang-2024}): 
\begin{align*}
    z(0,x) &= \alpha \, \mathrm{sech}\left(\frac{1}{\beta}\left({x} - \frac{1}{2L}\right)\right), \quad L=1, \; \\ 
    \alpha &\sim \mathrm{unif}(0.9, 1.1), \quad \beta \sim \mathrm{unif}(0.04, 0.06).
\end{align*}
The 100 trajectories are simulated for the case of zero control (that means $\bar{u} = u_d = 0$), and then with the robust control given by various values of $\lambda$ (the stabilizing control $\bar{u}$ is obtained from the dual EnKF). These simulations are repeated for various values of $\lambda$, $d_0$, both types of disturbances mentioned, and the values of $\nu$ mentioned. 


Recall that the control objective \eqref{eq:L2objective} in finite time is to drive $\|z(T,\cdot)\|_{L^2}$  as close to zero as possible. To illustrate the performance of the controllers we make two plots for both PDEs. 
The first plot is a heat map depicting the mean value (over the 100 simulations) of $\frac{\|z(T,\cdot)\|_{L^2}}{\|z(0,\cdot)\|_{L^2}}$ -- the ratio is plotted to highlight the order of magnitude by which $\|z(T,\cdot)\|_{L^2}$  has reduced relative to its initial value $\|z(0,\cdot)\|_{L^2}$. 
Moreover, in a second plot, we plot the mean and variance (over the 100 simulations) of  $\|z(t,\cdot)\|_{L^2}$ as a function of $t$ for $d_0 = 0.1$ and $\lambda = 0.2$. The trajectory is generated using three different control policies -- ``uncontrolled" trajectories are when the control is zero, ``optimal controlled" trajectories have $\lambda = 0$ (that is, only the optimal stabilizing control is applied and robust control is zero) and ``robust controlled" trajectories are with $\lambda = 0.2$. The plots for heat equation, Burgers with DMDc and Burgers with full nonlinear model are found in Figures \ref{fig:heat}, \ref{fig:dmdc}, and \ref{fig:full}, respectively. Other simulations details can be obtained in Appendix \ref{app:burgers}. 

For both PDEs, we observe that in the presence of disturbances, the robust control works better than the optimal control (control with $\lambda = 0$) in stabilizing the PDE. For the Burgers' equation in particular, the trajectories with robust control exhibit an order of magnitude lower value of $\|z(T,\cdot)\|_{L^2}$ compared to those with only the optimal control. Additionally, for the Burgers' equation, we also observe that the optimal control obtained using the full nonlinear model is significantly more effective than using the reduced-order DMDc model -- the controlled state settles much faster and closer to zero with the former than the latter. However, while the transient performance of the robust control is much better using the full nonlinear model, the settling performance is similar using both DMDc and the full nonlinear model. Thus the robust control term also compensates for model mismatch between DMDc and the full nonlinear model. The conclusions presented are true for all values of $\lambda$ chosen, both types of disturbances, and both values of viscosity. 



\begin{figure*}
\centering
    \begin{subfigure}[b]{0.45\textwidth}
        \centering
        \includegraphics[scale = 0.35]{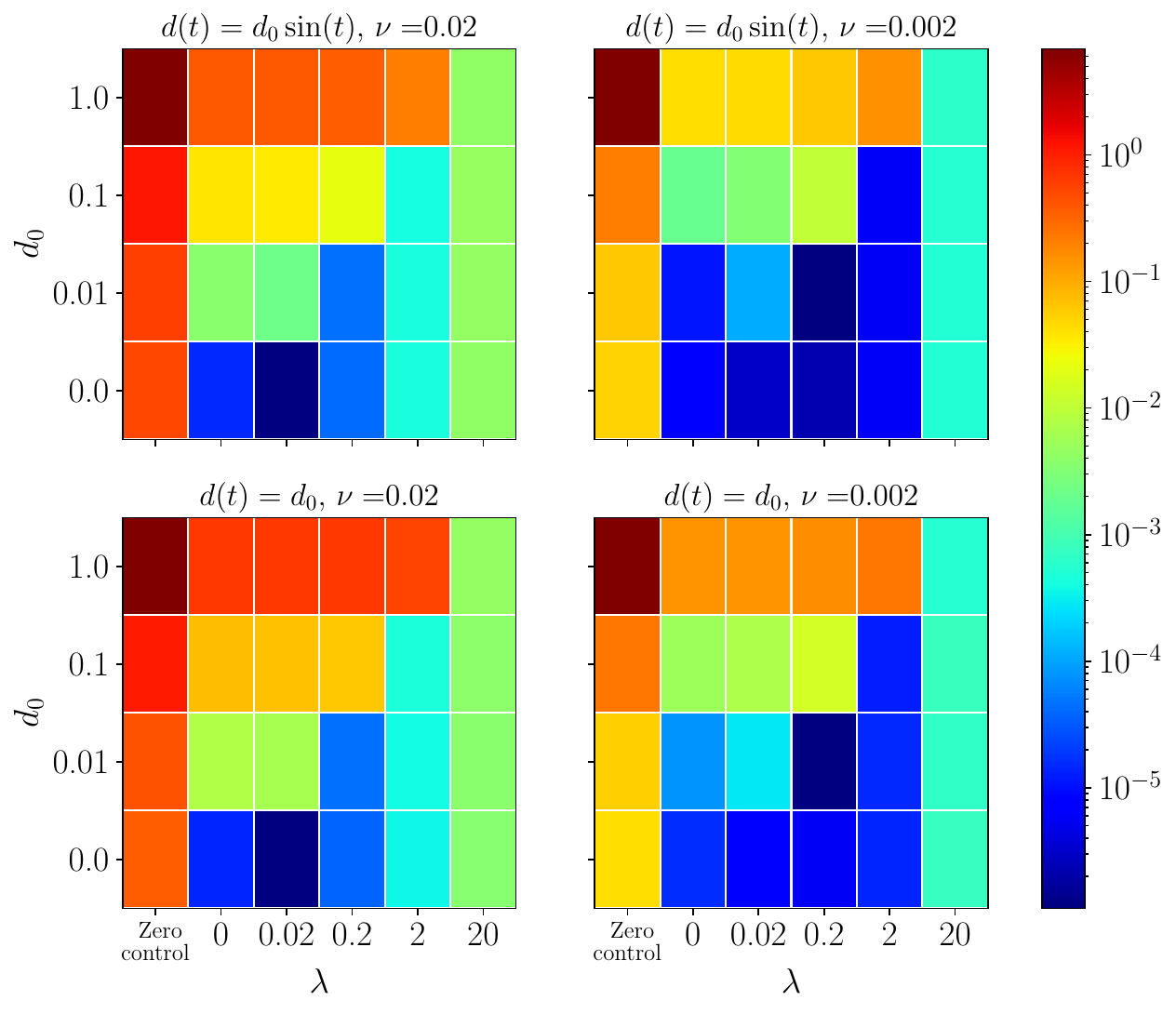}
        \caption{Mean of $\frac{\|z(T,\cdot)\|_{L^2}}{\|z(0,\cdot)\|_{L^2}}$ over 100 different simulations}
        \label{fig:heatmap-dmdc}
    \end{subfigure}
    \begin{subfigure}[b]{0.45\textwidth}
        \centering
        \includegraphics[scale = 0.35]{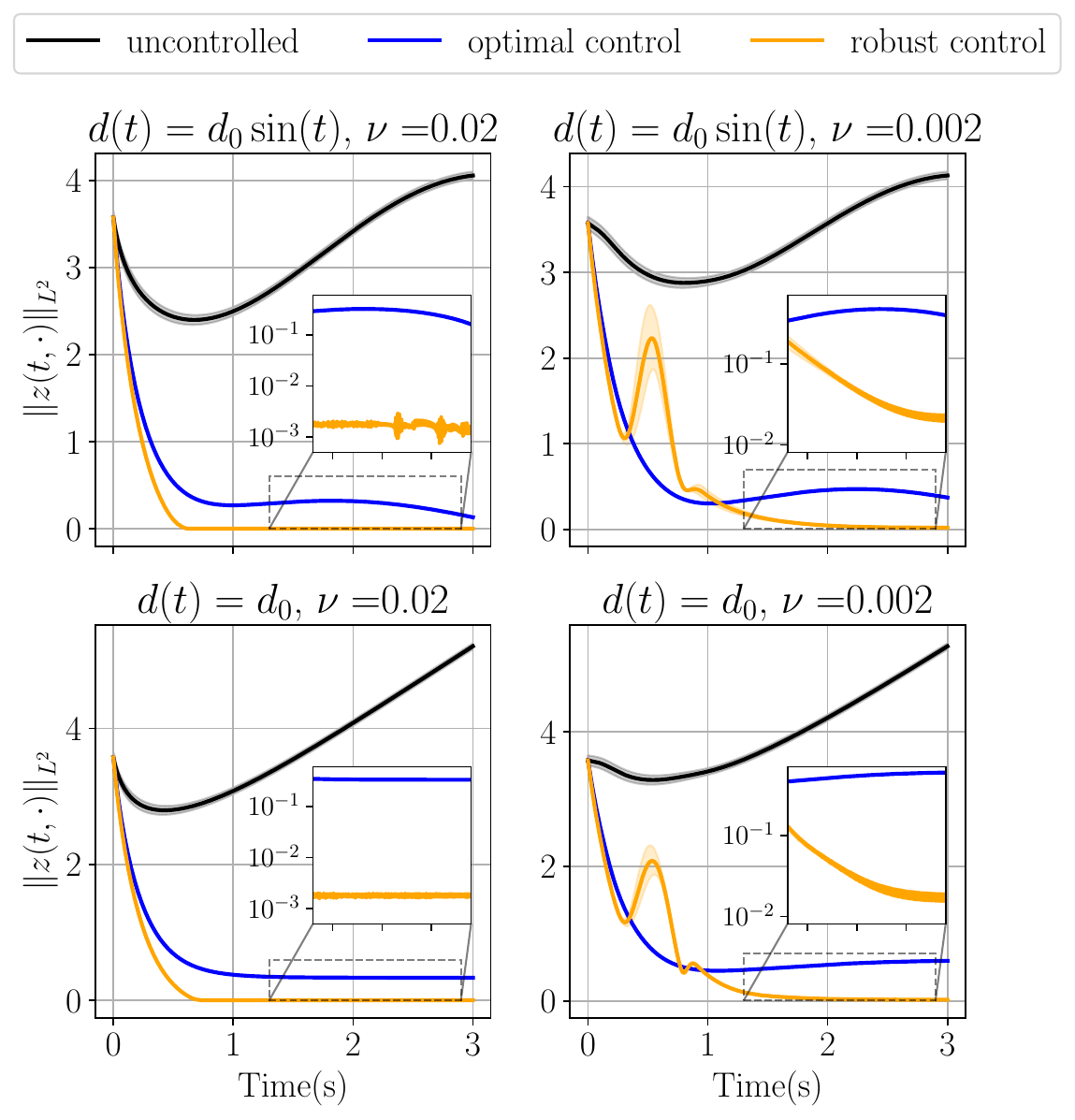}
        \caption{Mean and variance of $\|z(t,\cdot)\|_{L^2}$ for $d_0=0.1,\lambda=0.2$.}
        \label{fig:time-dmdc}
    \end{subfigure}
\caption{Results for control of Burgers equation using reduced order DMDc model.}
\label{fig:dmdc}
\end{figure*}

\begin{figure*}
\centering
    \begin{subfigure}[b]{0.45\textwidth}
        \centering
        \includegraphics[scale = 0.35]{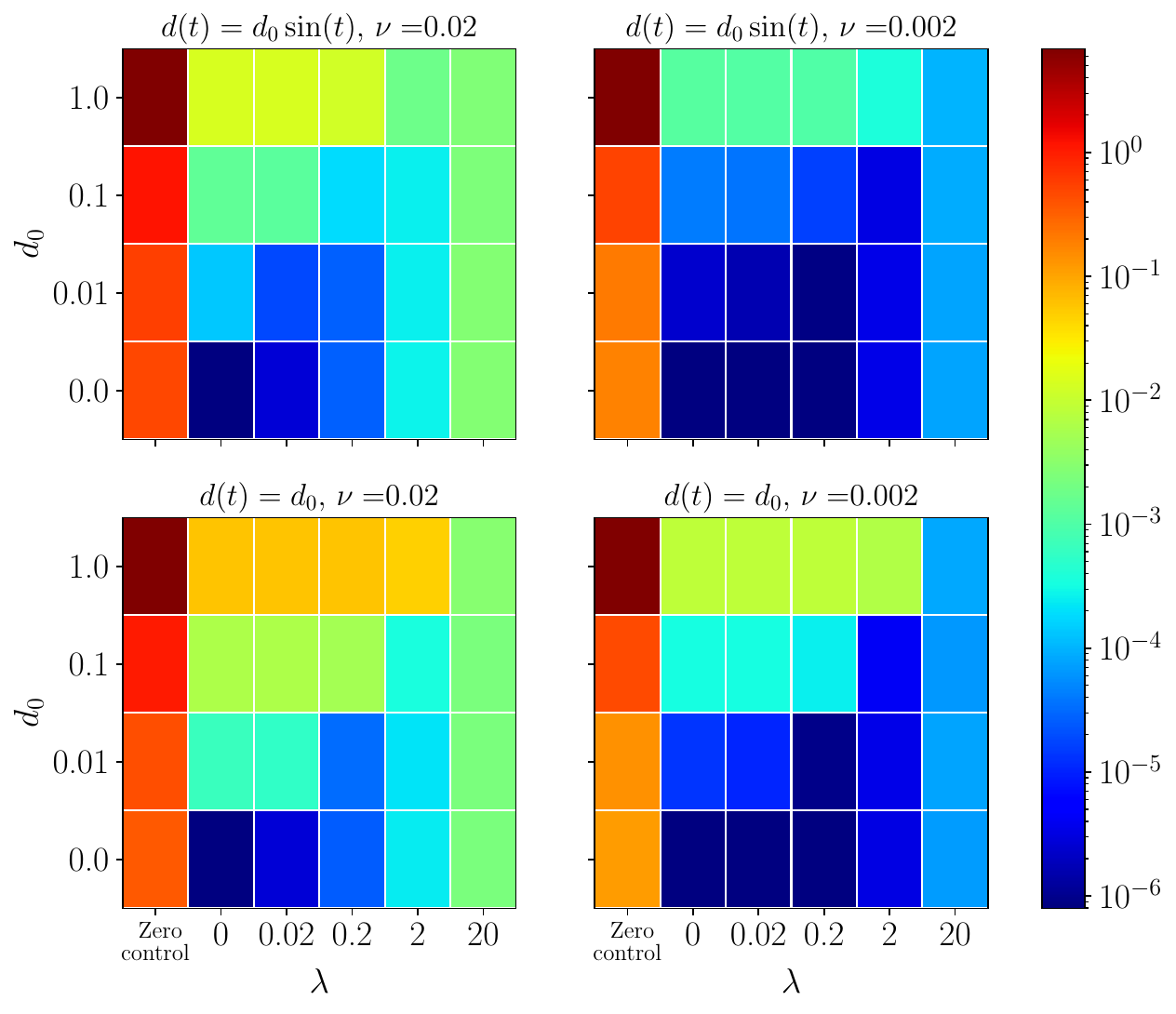}
        \caption{Mean of $\frac{\|z(T,\cdot)\|_{L^2}}{\|z(0,\cdot)\|_{L^2}}$ over 100 different simulations}
        \label{fig:heatmap-full}
    \end{subfigure}
    \begin{subfigure}[b]{0.45\textwidth}
        \centering
        \includegraphics[scale = 0.35]{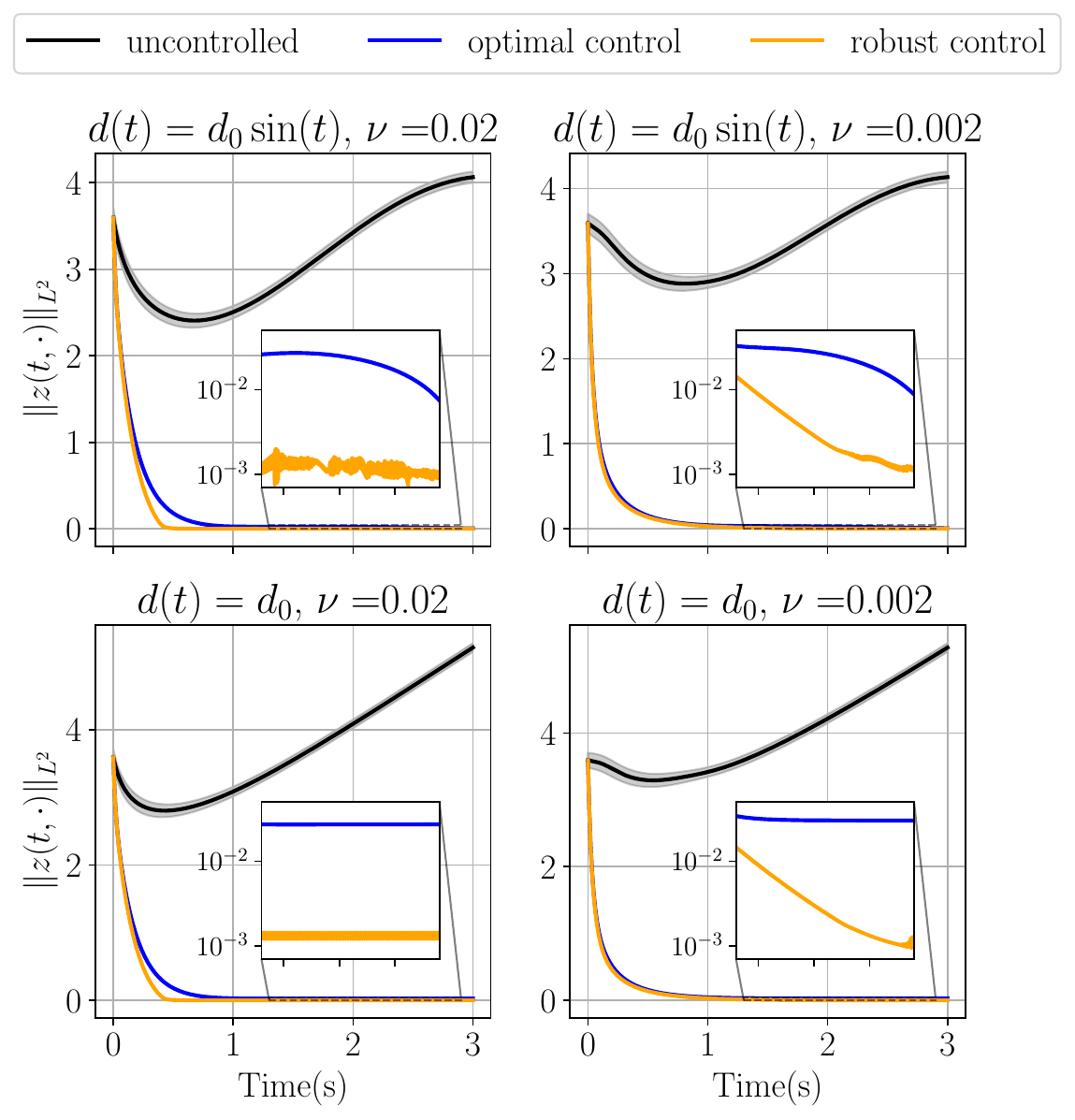}
        \caption{Mean and variance of $\|z(t,\cdot)\|_{L^2}$ for $d_0=0.1,\lambda=0.2$.}
        \label{fig:time-full}
    \end{subfigure}
\caption{Results for control of Burgers equation using full nonlinear model.}
\label{fig:full}
\end{figure*}


\section{Future work}
Some avenues for future work are: (i) to extend this approach to different PDEs such as the Allen-Cahn equation or the Korteweg deVries equation, (ii) to design a controller using information from sensors, thus turning a full state feedback problem into a partially observed problem (iii) to extend to the case when $d$ is Gaussian white noise. 

\bibliographystyle{siam} 
\bibliography{refs} 
\appendices

\section{Dual EnKF algorithm}
\label{app:enkf}
In this algorithm, we simulate over the time horizon $[0,T]$ an ensemble of $N$ particles $\{ Y_t^i : 1 \le i \le N, \quad 0 \le t \le T \}$ where the evolution equation for each particle is given by an Ito stochastic differential equation (SDE). 

\subsection{Linear system}
\label{app:linenkf}
The SDE \cite[Section 2]{joshi-2022} is given by
\begin{align*}
\ud {Y}^i_t &= {A {Y}^i_t \ud t + B\ud \backward{\eta}^i_t } + 
L^{(N)}_t \left(\frac{C{Y}^i_t+ C\hat{n}^{(N)}_t}{2} \right), \\ 
{Y}^i_T  &\stackrel{\text{i.i.d}}{\sim} \mathcal
N(0,S_T),\quad 1\leq i\leq N, \\
\end{align*}
${\eta}^i$ are an i.i.d Brownian motions with covariance $R\inv$ and $\hat{n}^{(N)}_t := \left( \frac{1}{N} \sum_{j=1}^{N} Y_t^j \right) $, and $L_t^{(N)} := S_t^{(N)}C\tp$ with
\begin{align*}
S^{(N)}_t := \frac{1}{N} \sum_{j=1}^{N} ({Y}_t^j - \hat{n}_t^{(N)}) ({Y}_t^j - \hat{n}_t^{(N)})\tp. 
\end{align*}
Finally, $\bar{P}^{(N)} := (S^{(N)}_0)\inv$.

\subsection{Nonlinear system}
\label{app:nlenkf}
The SDE \cite[Section 3]{joshi-2022} for the particle system is
\begin{align*}
\ud {Y}^i_t &= {a ({ Y}^i_t)\ud t +  b ({ Y}^i_t) \ud \backward{\eta}^i_t}  
 \\
 &+ \left(\sum_{j=1}^{N}(Y_t^j - n^{(N)}_t )(c({ Y}^j_t) - \hat{c}_t^{(N)}  )^\top\right) \frac{(c(z) + \hat{c}^{(N)})}{N-1} \\
{Y}^i_T  &\stackrel{\text{i.i.d}}{\sim} \mathcal
N(0,S_T),\quad 1\leq i\leq N, 
\end{align*}
and $\hat{c}^{(N)}_{t} := N^{-1} \sum_{i=1}^N c(Y^i_{t})$. Finally, $\grad\phi^{(N)}(x) := (S^{(N)}_0)\inv x$, where $S^{(N)}_t$ is the empirical covariance of $\{Y_t^i\}$, defined same as in the linear case. 

\subsection{Algorithm for Hamiltonian minimization }
\label{app:ubar}
Algorithm \ref{alg:ubar} minimizes Hamiltonian and calculates $\bar{u}$ for Section \ref{sec:linear}. To minimize Hamiltonian in Section \ref{sec:nonlinear}, in Algorithm \ref{alg:ubar} make the following two changes: replace $\bar{P}^{(N)}$  by $\grad\phi^{(N)}$ and use the Hamiltonian defined in \eqref{eq:ham-nl}.
\begin{algorithm}
\begin{algorithmic}[1]
\Input System state $x$, $\bar{P}^{(N)}$, $Q,R$, Hamiltonian definition from \eqref{eq:ham-lin}, $\{e_i\}_{i=1}^{m}$ the standard basis of $\R^m$
\If{$B$ is known}
\State \Return $\bar{u}^{(N)} := -R\inv B\tp \bar{P}^{(N)}x$
\ElsIf{$B$ is unknown}
\For{i=1,2,\ldots,m}
    \State $(\bar{u}^{(N)})_i := H(x,R\inv e_i) - H(x,0) - \half (R\inv)_{ii}$
    \EndFor
\EndIf
\end{algorithmic}
\caption{Algorithm for Hamiltonian minimization}
\label{alg:ubar}
\end{algorithm}

\section{Simulation details}
\subsection{Heat equation}
\label{app:heat}
Simulation parameters are as follows. The simulation time $T = 0.1$ 
with simulation time step = 0.001. 
The number of states of the discretized PDE \eqref{eq:discrete-pde} is $p=100$. 
The number of control basis functions is $m=8$, and they are $\chi_j$ is the indicator functions of $[\frac{i}{10},\frac{(i+1)}{10}]$. The regularization parameter for robust control is $r = 0.002$.
The matrices  $Q = \id$, $G=\id$, $R=\id$ for both, the full nonlinear control and control using DMDc model.
The number of dual EnKF particles is chosen as $N=10000$. The controlgym library \cite{zhang2023controlgym} is used for PDE simulation. 

\subsection{Burgers equation}
\label{app:burgers}
Simulation parameters are as follows. The simulation time $T = 3$ 
with simulation time step = 0.001. 
The number of states of the discretized PDE \eqref{eq:discrete-pde} is $p=128$. 
The number of control basis functions is $m=10$, and they are $\chi_j$ is the indicator functions of $[\frac{i}{10},\frac{(i+1)}{10}]$. The regularization parameter for robust control is $r = 0.002$.
The number of states in the DMDc model \eqref{eq:DMDc} is $n=10$. The matrices  $Q = \id$, $G=\id$, $R=0.1\id$ for both, the full nonlinear control and control using DMDc model.
The number of dual EnKF particles is chosen as $N=1000$. 
The controlgym library \cite{zhang2023controlgym} is used for PDE simulation.

\end{document}